\documentclass{amsart}
\usepackage{color}
\numberwithin{equation}{section}
\theoremstyle{definition}
\newtheorem{thm}{Theorem}[section]
\newtheorem{prop}[thm]{Proposition}
\newtheorem{df}[thm]{Definition}
\newtheorem{lem}[thm]{Lemma}
\newtheorem{rem}[thm]{Remark}
\newtheorem{cor}[thm]{Corollary}

\newtheorem{claim}[thm]{Claim}
\newtheorem*{ack}{Acknowledgments}
\newtheorem*{mt}{Main Theorem}
\def\Aut{\mathop{\mathrm{Aut}}\nolimits}
\def\Lin{\mathop{\mathrm{Lin}}\nolimits}
\def\Bir{\mathop{\mathrm{Bir}}\nolimits}
\def\ord{\mathop{\mathrm{ord}}\nolimits}

\def\im{\mathop{\mathrm{Im}}\nolimits}

\def\Hom{\mathop{\mathrm{Hom}}\nolimits}

\def\id{\mathop{\mathrm{id}}\nolimits}
\title[Quasi-Galois points for quartic surfaces]
{Quasi-Galois points for quartic surfaces}
\author[K.~Miura]{Kei Miura}
\address{Department of Applied Science, Yamaguchi University, 
2-16-1 Tokiwadai, Ube, Yamaguchi 755- 8611, Japan}
\email{kmiura@yamaguchi-u.ac.jp}
\author[S.~Taki]{Shingo Taki}
\address{Department of Mathematics, Tokai University,
4-1-1, Kitakaname, Hiratsuka, Kanagawa, 259-1292, Japan}
\email{staki@tokai.ac.jp}
\urladdr{https://taki.sm.u-tokai.ac.jp}
\date{\today}
\dedicatory{}
\thanks{}
\begin{document}

\begin{abstract}
We study quasi-Galois points, which are a generalization of Galois points.
We characterize smooth quartic surfaces admitting a quasi-Galois point in terms of $K3$ surfaces with a certain involution,
and provide a criterion for quasi-Galois points for quartic surfaces.
Moreover, we also determine all quasi-Galois points of the Fermat quartic surface and show that it has exactly 28 quasi-Galois points.
\end{abstract}

\maketitle

\section{Introduction}\label{Introduction}
We will work over $\mathbb{C}$, the field of complex numbers, throughout this paper.
In algebraic geometry, hypersurfaces of a projective space are one of the most fundamental objects.
H.~ Yoshihara introduced a remarkable new notion, that is, a Galois point, in 1996.
Galois points have been studied and generalized by many mathematicians.
For example, S.~Fukasawa, the first-named author, and T.~Takahashi 
introduced the concept of quasi-Galois points. 
In \cite{FMT} \cite{FMT2}, 
the authors investigated quasi-Galois points for smooth plane curves in detail.

Let $V$ be a (smooth) hypersurface in $\mathbb{P}^{n}$
and $\mathbb{C}(V)$ the function field of $V$.
For a point $P \in \mathbb{P}^{n}$, we consider a 
projection $\pi_{P}: V \dashrightarrow \mathbb{P}^{n-1}$.
Note that the projection $\pi_{P}$ induces the extension 
$\mathbb{C}(V)/\pi_{P}^{\ast}\mathbb{C}(\mathbb{P}^{n-1})$ of function fields.
If the extension is Galois then the point $P$ is called a \textit{Galois point} for $V$.
It is known that an element of its Galois group acts on $V$ as a birational automorphism, and preserves $\pi_{P}$.
We consider the subset $G[P]:=\{ \tau \in \Bir (V)  \mid \pi_{P} \circ \tau = \pi_{P} \}$
of all birational automorphisms of $V$.
If the order of $G[P]$ satisfies $|G[P]| \geq 2$ then $P$ is called \textit{quasi-Galois}.
In particular, $|G[P]|=\deg \pi_{P}$ holds if and only if $P$ is Galois.
In this sense, quasi-Galois points are a generalization of Galois points.

Here is a fundamental problem about (quasi-)Galois points:
\begin{itemize}
\item Find the equation of a quartic surface with (quasi-)Galois points. 
\item Find the number of (quasi-)Galois points of a given hypersurface.
\item Characterize hypersurfaces with (quasi-)Galois points.
\end{itemize}

This problem has been solved for quartic surfaces.
Yoshihara \cite{Yoshihara-GK3} \cite{Yoshihara--hyper} determined equations of quartic surfaces with Galois points,
and the number of Galois points.
Recently, the authors \cite{MT1} \cite{MT2} obtained a characterization of quartic surfaces with a Galois point 
in terms of $K3$ surfaces with an automorphism.

In this paper, we consider this problem in the case of quasi-Galois points for quartic surfaces.
The following is the main theorem of this paper.
(For detailed terminology, refer to Section 2 and beyond.)
\begin{mt}
The following holds:
\begin{enumerate}
\item\label{mt1} 
There exists a one-to-one correspondence between 
the projective equivalence classes of pairs $(S, P)$, 
where $S$ is a smooth quartic surface and $P$ is a quasi-Galois point for $S$ which is not Galois, 
and 
the isomorphism classes of general 2-elementary $K3$ surfaces 
$(X, \iota)$ of type $(8,8,1)$, via a quasi-Galois embedding.

\item\label{mt2} 
Let $S$ be a smooth quartic surface defined by the quartic defining polynomial $F$,
$D_{P}F=0$ and $D^{3}_{P}F=0$ be the 1-st polar hypersurface 
and the 3-rd polar hypersurface of $S$ at a point $P:=[a:b:c:d] \in \mathbb{P}^{3}$ 
which does not contain $S$, respectively.
Then $P$ is a quasi-Galois point if and only if $D^{3}_{P}F$ divides $D_{P}F$.

\item\label{mt3} 
The Fermat quartic surface admits 28 quasi-Galois points.
\end{enumerate}
\end{mt}

We summarize the contents of this paper. 
The section \ref{fund-qgp} is preliminary.
We review some basic results about quasi-Galois points.
Here, we also provide some supplementary results to \cite{FMT2}.
In Section \ref{qsvs2ele}, we discuss Main Theorem (\ref{mt1}) and (\ref{mt2}).
The Section \ref{G-pairs} introduces the concept of $G$-pairs, which are defined in \cite{FMT2}.
Then we determine the general equation of a quartic surface with (at least) 2 quasi-Galois points.
In Section \ref{Fermat}, we study special quartic surfaces and discuss Main Theorem (\ref{mt3}).

\begin{rem}
The first-named author and Takahashi \cite{Miura-Takahashi} obtained results on the number of 
Galois and quasi-Galois points for Fermat hypersurfaces 
around the same time as the present work, using a different method. 
\end{rem}

\begin{ack}
This work was supported by JSPS KAKENHI Grant Numbers JP18K03230 and JP23K03036.
\end{ack}

\section{Quasi-Galois points}\label{fund-qgp}
In this section, we see basic results about quasi-Galois points.

Let $V$ be an irreducible hypersurface in $\mathbb{P}^{n}$
and $\mathbb{C}(V)$ the function field of $V$.
For a point $P \in \mathbb{P}^{n}$, 
we consider a projection $\pi_{P}: V \dashrightarrow \mathbb{P}^{n-1}$ from $P$,
where the $\mathbb{P}^{n-1}$ is a hyperplane in $\mathbb{P}^{n}$ which does not contain $P$.
Let $ \Bir (V) $ be the birational automorphism group of $V$, and put $G[P]:=\{ \tau \in \Bir (V)  \mid \pi_{P} \circ \tau = \pi_{P} \}$.
If the order of $G[P]$ satisfies $|G[P]| \geq 2$ then $P$ is called \textit{quasi-Galois}.

\begin{rem}
Let $\pi_{P}$ and $\pi_{P}'$ be projections from $P$.
Since there exists a projective transformation $f : \mathbb{P}^{n-1} \to \mathbb{P}^{n-1}$ satisfying $f \circ \pi_{P} = \pi_{P}'$,
if $\tau \in \Bir (V)$ satisfies $\pi_{P} \circ \tau = \pi_{P}$ then we have
\[ \pi_{P}' \circ \tau=(f \circ \pi_{P})\circ \tau= f \circ (\pi_{P}\circ \tau )= \pi_{P}'. \]
Hence, the structure of $G[P]$ does not depend on the choice of $\pi_{P}$ but only on $P$.
\end{rem}

\begin{lem}
The set $G[P]$ is a subgroup of  $\Bir (V)$ such that 
the order $|G[P]|$ divides the degree of $\pi_{P}$.
\end{lem}
\begin{proof}
The fixed field $\mathbb{C}(V)^{G[P]}$ is an intermediate field of
the extension $\mathbb{C}(V)/\pi_{P}^{\ast}\mathbb{C}(\mathbb{P}^{n-1})$.
Thus $|G[P]|$ divides $\deg \pi_{P}$.
\end{proof}

\begin{rem}
The condition $|G[P]|=\deg \pi_{P}$ holds if and only if $P$ is a Galois point, hence
quasi-Galois points are a generalization of Galois points.

In the following, we assume that $n=3, \deg V=4$ (i.e., $V$ is a smooth quartic surface).
In this case, a quasi-Galois point contained in $V$ is a Galois point.
It is settled in \cite{Yoshihara-GK3}.
Thus, a quasi-Galois point for $V$ does not mean a Galois point contained in $V$ (i.e., an inner Galois point) in this paper.
\end{rem}

It is well known that there exists an automorphism of quartic surfaces
which is not induced by a projective transformation \cite{MatsumuraMonsky}. 
We have the following.

\begin{lem}\label{geneprojective}
Let $S$ be a smooth quartic surface with a quasi-Galois point $P$.
If $\tau \in G[P]$ then $\tau$ is a projective transformation.
\end{lem}

\begin{proof}
It is sufficient to show that $\tau$ maps linear expressions to linear expressions,
that is to say, it preserves the polarization.
It follows from 
$\tau ^{\ast} (\pi_{P}^{\ast}\mathcal{O}_{\mathbb{P}^{2}}(1))
=(\pi_{P}\circ \tau )^{\ast}\mathcal{O}_{\mathbb{P}^{2}}(1)
=\pi_{P}^{\ast}\mathcal{O}_{\mathbb{P}^{2}}(1)$.
\end{proof}

\begin{rem}
This lemma can also be proved by an argument similar to that in \cite{Yoshihara-GK3} and \cite{Yoshihara-sextic}.
The first-named author and Takahashi \cite[Theorem 1.6]{Miura-Takahashi} give similar results for the more general case.
\end{rem}

\begin{cor}\label{number-finite}
The number of quasi-Galois points for a smooth quartic surface $S$ is finite.
\end{cor}
\begin{proof}
Since the group $\Lin (S):=\Aut (S) \cap PGL(4, \mathbb{C})$
is finite by \cite[Theorem 1]{MatsumuraMonsky}, the assertion holds.
\end{proof}

\subsection{Quasi-Galois points for quartic curves}

We recall quasi-Galois points for quartic curves.
In \cite{FMT, FMT2}, the number of quasi-Galois points for smooth plane quartic curves 
was already completely determined.
Here, we explicitly give the precise locations of these quasi-Galois points.

For a smooth quartic curve or a smooth quartic surface, 
we denote $\delta '[4]$ as the number of Galois points, and 
$\delta'[2]$ as the number of quasi-Galois points which are not Galois.

\begin{lem}\label{1parameter3-curve}
Let $C$ be a smooth quartic curve 
\[ X^{4}+Y^{4}+Z^{4}+c_{1}(X^{2}Y^{2}+Y^{2}Z^{2}+Z^{2}X^{2})=0,\]
where $c_{1}\in \mathbb{C}\setminus \{0, -1\}$ and $c_{1}^{2}+3c_{1}+18 \neq 0$.
Then we have $\delta' [2]=9$ and 
quasi-Galois points are $[1:0:0]$, $[0:1:0]$, $[0:0:1]$
$[1:(-1)^{i}:0]$, $[1:0:(-1)^{i}]$ and  $[0:1:(-1)^{i}]$ ($i=0, 1$).
\end{lem}
\begin{proof}
By \cite[Theorem 5.12 (3)]{FMT2}, we have $\delta' [2]=9$.
For $P:=[1:0:0]$, we take the following involution of $C$:
\[\tau_{P}:[X:Y:Z] \mapsto [-X:Y:Z].\]
Then a generator of $G[P]$ is given by $\tau_{P}$, hence $P$ is quasi-Galois.
For $Q:=[1:(-1)^{i}:0]$, we take the following involution of $C$:
\[\tau_{Q}:[X:Y:Z] \mapsto [Y:X:(-1)^{i+1}Z].\]
Then a generator of $G[Q]$ is given by $\tau_{Q}$, hence $Q$ is quasi-Galois.

A similar assertion holds in the other points.
\end{proof}

If $c_{1}^{2}+3c_{1}+18 = 0$ then $C$ satisfies $\delta' [2]=21(=9+12)$ by \cite[Proposition 5.11 (1)]{FMT2}.

\begin{lem}\label{Kleincurve}
Assume $c_{1}^{2}+3c_{1}+18 = 0$ holds.
Let $C$ be a smooth quartic curve 
\[ X^{4}+Y^{4}+Z^{4}+c_{1}(X^{2}Y^{2}+Y^{2}Z^{2}+Z^{2}X^{2})=0.\]
Then quasi-Galois points for $C$ are 9 points in Lemma \ref{1parameter3-curve}
and 12 points 
$[1:\pm \lambda/2: \pm  \lambda/2], [\pm \lambda/2: 1: \pm  \lambda/2], [\pm \lambda/2: \pm  \lambda/2:1]$
(with all combinations of signs),
where $ \lambda$ is a solution of $ \lambda^{2} =4c_{1}/(6-c_{1})$.
Thus, we have $\delta' [2]=21$.
\end{lem}
\begin{proof}
By \cite[Proposition 5.11 (1)]{FMT2}, we have $\delta' [2]=21$.
We shall show that the point $P=[1:- \lambda/2:  \lambda/2]$ is quasi-Galois.
We remark that the projective transformation $\tau$ which is represented by the matrix 
$\begin{pmatrix}
 0 & 2/\lambda & -2/\lambda \\
\lambda & 1 & 1 \\
-\lambda & 1 & 1 
\end{pmatrix}$
acts on $C$ as an involution.
Since $\pi_{P}$ is given by $[X:Y:Z] \mapsto [Y+\frac{\lambda}{2}X:Z-\frac{\lambda}{2}X]$,
we have  
\begin{align*}
(\pi_{P}\circ \tau) \left([X:Y:Z]\right) 
&=\pi_{P} \left( \left[\frac{2}{\lambda}Y-\frac{2}{\lambda}Z: \lambda X+Y+Z :-\lambda X+Y+Z \right ] \right)\\
&=\left[ \lambda X+Y+Z+\frac{\lambda}{2}\left( \frac{2}{\lambda}Y-\frac{2}{\lambda}Z \right)
:-\lambda X+Y+Z-\frac{\lambda}{2}\left( \frac{2}{\lambda}Y-\frac{2}{\lambda}Z \right) \right]\\
&=[ \lambda X+2Y:-\lambda X+2Z]\\
&=\pi_{P} \left([X:Y:Z]\right)
\end{align*}
for all $[X:Y:Z] \in C$. Thus $P$ is quasi-Galois.
A similar assertion holds in the other points.
\end{proof}

\begin{lem}\label{1parameter2-curve}
Let $C$ be a smooth quartic curve 
\[ X^{4}+Y^{4}+Z^{4}+c_{1}X^{2}Y^{2}=0,\]
where $c_{1}\in \mathbb{C}^{\ast}$ and $c_{1}\neq \pm 6$.
Then we have $\delta' [2]=6$ and 
quasi-Galois points are $[1:0:0]$, $[0:1:0]$, $[0:0:1]$ and 
$[1:\zeta_{4}^{i}:0]$ ($i=0, 1, 2, 3$) where $\zeta_{4}$ is a primitive 4-th root of unity.
In particular, a point $[0:0:1]$ is Galois.
\end{lem}
\begin{proof}
It is the same as Lemma \ref{1parameter3-curve}.
We remark that a generator of $G[[1:\zeta_{4}^{i}:0]]$ is given by
$[X:Y:Z] \to [(-1)^{i}Y: X: \zeta_{4}^{i}Z]$.
See also \cite[Theorem 5.12 (4)]{FMT2} and \cite[Corollary 6]{Yoshihara--hyper}.
\end{proof}

\begin{lem}\label{2parameter-curve}
Let $C$ be a smooth quartic curve 
\[ X^{4}+Y^{4}+Z^{4}+c_{1}(X^{2}Y^{2}+X^{2}Z^{2})+c_{2}Y^{2}Z^{2}=0,\]
where $c_{1}, c_{2} \in \mathbb{C}$, $c_{1}\neq 0$ and $c_{1} \neq \pm c_{2}$.
Then we have $\delta' [2]=5$ and 
quasi-Galois points are $[1:0:0]$, $[0:1:0]$, $[0:0:1]$ and  $[0:1:(-1)^{i}]$ ($i=0, 1$).
\end{lem}
\begin{proof}
It is the same as Lemma \ref{1parameter3-curve}.
We remark that a generator of $G[[0:1:(-1)^{i}]]$ is  given by $[X:Y:Z] \to [(-1)^{i+1}X:Z:Y]$. 
See also \cite[Theorem 5.12 (5)]{FMT2}.
\end{proof}

\section{Quartic surfaces with a quasi-Galois point and 2-elementary $K3$ surfaces}\label{qsvs2ele}

In this section, we determine the general equation of a smooth 
quartic surface with a quasi-Galois point.
Furthermore, we characterize such quartic surfaces as $K3$ surfaces.

\begin{prop}\label{equation1ko}
If a smooth quartic surface $S$ has a quasi-Galois point which is not Galois 
then its homogeneous equation is given by
$X^{4}+X^{2}F_{2}(Y,Z,W)+F_{4}(Y, Z, W)=0$
(up to projective equivalence).
Here $F_{d}$ is a homogeneous polynomial of degree $d$.
Moreover the quasi-Galois point is $[1:0:0:0]$ and a generator of $G[[1:0:0:0]]$ is 
given by $[X:Y:Z:W] \mapsto [-X:Y:Z:W]$.
\end{prop}
\begin{proof}
Let $P=[1:a:b:c] \in \mathbb{P}^{3}$ be a quasi-Galois point which is not Galois for $S$.
Since $\deg \pi_{P} \leq \deg S =4$ holds,  
a generator $\tau$ of $G[P]$ is of order 2.
Hence, we may assume that $\tau$ is
\[
\begin{pmatrix} 
-1 & 0 & 0 & 0 \\
0 & 1 & 0 & 0 \\
0 & 0 & 1 & 0 \\
0 & 0 & 0 & 1 
 \end{pmatrix}
\ \text{or} \ 
\begin{pmatrix} 
-1 & 0 & 0 & 0 \\
0 & -1 & 0 & 0 \\
0 & 0 & 1 & 0 \\
0 & 0 & 0 & 1 
 \end{pmatrix}
\]
by replacing coordinates of $\mathbb{P}^{3}$.

Suppose an automorphism $\tau$ of $S$ is given by the latter, hence
\[ \tau : [X:Y:Z:W] \mapsto [-X:-Y:Z:W].\]
Since $\pi_{P}$ is given by $[X:Y:Z:W] \mapsto [Y-aX:Z-bX:W-cX]$, 
we have \[\pi_{P} \circ \tau : [X:Y:Z:W] \mapsto [-Y+aX:Z+bX:W+cX]. \]
Since $P$ is quasi-Galois, we have \[[-Y+aX:Z+bX:W+cX]= [Y-aX:Z-bX:W-cX].\]
This implies that there exists a non-zero constant $k \in \mathbb{C}^{\ast}$ such that
\[\begin{pmatrix} -Y+aX \\ Z+bX \\ W+cX \\  \end{pmatrix}
=k\begin{pmatrix} Y-aX \\ Z-bX \\ W-cX \\  \end{pmatrix},\]
hence we have
\[\begin{pmatrix} 
a(1+k) & -1-k & 0 & 0 \\
b(1+k) & 0 & 1-k & 0 \\
c(1+k) & 0 & 0 & 1-k 
 \end{pmatrix}
\begin{pmatrix} X \\ Y \\ Z \\ W \\  \end{pmatrix}
=\begin{pmatrix} 0 \\ 0 \\ 0 \\  \end{pmatrix}.\]
But the coefficient matrix of the system of linear equations is not zero
for any $a$, $b$, $c$, and $k$.
This is a contradiction.

If $\tau$ is given by 
\[ \tau : [X:Y:Z:W] \mapsto [-X:Y:Z:W]\]
then we have \[[Y+aX:Z+bX:W+cX]= [Y-aX:Z-bX:W-cX].\]
Thus we have $P=[1:a:b:c]=[1:0:0:0]$.
Moreover, since $\tau$  is an automorphism of a smooth quartic surface $S$,
its general defining equation is
$X^{4}+X^{2}F_{2}(Y,Z,W)+F_{4}(Y, Z, W)=0$.
\end{proof}

\begin{rem}
If a smooth quartic surface $S$ has an outer Galois point  
then it is of the form: $X^{4}+F_{4}(Y, Z, W)=0$ (up to projective equivalence)
by \cite[Corollary 6]{Yoshihara--hyper}.
\end{rem}

\begin{cor}\label{genefixqg}
Let $S$ be a smooth quartic surface with a quasi-Galois point $P$ which is not Galois.
If $\tau \in G[P]$ then $P\not \in S$ is a fixed point of $\tau$ as an automorphism of $\mathbb{P}^{3}$.
\end{cor}

As above, a quasi-Galois point induces a special automorphism, and vice versa.

\begin{prop}\label{hobokore}
Let $S$ be a smooth quartic surface.
If  the projective transformation
\[\tau : [X:Y:Z:W] \mapsto [(-1)^{i}Y:X:\zeta_{4}^{2-i}Z:\zeta_{4}^{2-i}W]\]
of order 2 acts on $S$ as an involution then points
$P_{i}:=[1:\zeta_{4}^{i}:0:0]$ ($i=0, 1,2, 3$) are quasi-Galois points for $S$.
\end{prop}
\begin{proof}
We consider a projection
\[\pi_{P_{i}}: [X:Y:Z:W] \mapsto [Y-\zeta_{4}^{i}X:Z:W].\]
For all points $[X:Y:Z:W] \in S$, we have 
\begin{align*}
(\pi_{P_{i}}\circ \tau )\left([X:Y:Z:W]\right) &=\pi_{P_{i}} \left( [(-1)^{i}Y:X:\zeta_{4}^{2-i}Z:\zeta_{4}^{2-i}W] \right)\\
&=[X-(-1)^{i}\zeta_{4}^{i}Y:\zeta_{4}^{2-i}Z:\zeta_{4}^{2-i}W]\\
&=[-\zeta_{4}^{i}X+(-1)^{i}\zeta_{4}^{2i}Y:-\zeta_{4}^{2}Z:-\zeta_{4}^{2}W]\\
&=[-\zeta_{4}^{i}X+(-1)^{i}(-1)^{i}Y:Z:W]\\
&=\pi_{P_{i}} \left([X:Y:Z:W]\right).
\end{align*}
Thus $P_{i}$ is quasi-Galois for $S$.
\end{proof}

Whether a given point is quasi-Galois can be determined as follows.
\begin{prop}\label{polar-hyper}
Let $S$ be a smooth quartic surface defined by the quartic defining polynomial $F$,
$D_{P}F=0$ and $D^{3}_{P}F=0$ be the 1-st polar hypersurface 
and the 3-rd polar hypersurface of $S$ at $P:=[a:b:c:d] \in \mathbb{P}^{3} \setminus S$, respectively.
Hence
\begin{align*}
D_{P}F &=a\frac{\partial F}{\partial X}+b\frac{\partial F}{\partial Y}+c\frac{\partial F}{\partial Z}+d\frac{\partial F}{\partial W},\\
D^{3}_{P}F &=\sum_{i+j+k+l=3}
\frac{3!}{i! j! k! l!}
a^{i} b^{j} c^{k} d^{l}
\frac{\partial^{3} F}{\partial X^{i} \partial Y^{j} \partial Z^{k} \partial W^{l}}.
\end{align*}
Then $P$ is a quasi-Galois point if and only if there exists a homogeneous polynomial $Q_{2}(X, Y, Z, W)$ of degree 2 such that 
$D_{P}F=D^{3}_{P}F\cdot Q_{2}(X, Y, Z, W)$ holds.
\end{prop}
\begin{proof}
By replacing coordinates, we may put $P=[1:0:0:0]$.
Let $[X_{0}: X_{1}: X_{2}: X_{3}]$ be the new coordinate.
Then $S$ is of the form
\[ S:X_{0}^{4}+A_{1}(X_{1},X_{2}, X_{3})X_{0}^{3}+A_{2}(X_{1},X_{2}, X_{3})X_{0}^{2}+A_{3}(X_{1},X_{2}, X_{3})X_{0}+A_{4}(X_{1},X_{2}, X_{3})=0\]
where $A_{d}(X_{1},X_{2}, X_{3})$ is a homogeneous polynomial of degree $d$.
Moreover we have 
\begin{align*}
D_{P}F &=4X_{0}^{3}+3A_{1}X_{0}^{2}+2A_{2}X_{0}+A_{3},\\
D^{3}_{P}F &=24X_{0}+6A_{1}.
\end{align*}
We remark that $D_{P}F=D^{3}_{P}F \cdot Q_{2}$ holds if and only if $D_{P}F(-A_{1}/4)=0$  holds %$A_{3}=A_{1}A_{2}/2-A_{1}^{3}/8$ 
by the factor theorem.

Put $Y_{0}:=X_{0}+A_{1}/4$. 
Note that $P$ is represented by $[1:0:0:0]$ in the new coordinate system $[Y_{0}:X_{1}:X_{2}:X_{3}]$ as well.
Then $S$ is given by 
\[ S:Y_{0}^{4}+B_{1}(X_{1},X_{2}, X_{3})Y_{0}^{3}+B_{2}(X_{1},X_{2}, X_{3})Y_{0}^{2}+B_{3}(X_{1},X_{2}, X_{3})Y_{0}+B_{4}(X_{1},X_{2}, X_{3})=0\]
where $B_{d}(X_{1},X_{2}, X_{3})$ is a homogeneous polynomial of degree $d$.
By the Taylor expansion, we have
\begin{align*}
B_{1}&=\left . \frac{1}{3!}\frac{\partial^{3} F}{\partial X_{0}^{3}} \right |_{X_{0}=-A_{1}/4}\\
&=\left .(4X_{0}+A_{1}) \right |_{X_{0}=-A_{1}/4}\\
&=0, \\
B_{3}&=\left . \frac{\partial F}{\partial X_{0}} \right |_{X_{0}=-A_{1}/4}\\
&=D_{P}F(-A_{1}/4).
\end{align*}

Since if $D_{P}F=D^{3}_{P}F\cdot Q_{2}(X, Y, Z, W)$ holds then $D_{P}F(-A_{1}/4)=0$,
$S$ is of the form $S: Y_{0}^{4}+B_{2}Y_{0}^{2}+B_{4}=0$.
By Proposition \ref{equation1ko}, $P$ is a quasi-Galois point.
The same applies in reverse.
\end{proof}

\subsection{2-elementary $K3$ surfaces}

We recall basic properties of 2-elementary $K3$ surfaces. See also \cite{Ni3} for details.

A \textit{lattice} $L$ is a free abelian group of finite rank $r$ equipped with 
a non-degenerate symmetric bilinear form, which will be denoted by $\langle \ , \ \rangle $.
The bilinear form $\langle \ , \ \rangle $ determines a 
canonical embedding $L\subset L^{\ast }=\Hom (L,\mathbb{Z})$. 
We denote by $A_{L}$ the factor group $L^{\ast }/L$ which is a finite abelian group.
A lattice $L$ is called \textit{2-elementary} if $A_{L}\simeq (\mathbb{Z}/2\mathbb{Z})^{\oplus a}$,
where $a$ is the minimal number of generator of $A_{L}$.
In the following, we treat only even indefinite lattices. 
Hence for any $x \in L$, the value of $\langle x , x \rangle$ is even and 
the signature of $\langle \ , \ \rangle$ is indefinite.

\begin{df}
Let $L:=(L, \langle \ , \ \rangle )$ be an even 2-elementary lattice and 
$(L^{\ast}, \langle \ , \ \rangle _{L^{\ast}})$ its dual lattice.
These induce the discriminant form 
$q : A_{L} \to \mathbb{Q}/2\mathbb{Z}, \  q(x+L) = \langle x, x \rangle _{L^{\ast}}+2\mathbb{Z}$.
Then we put
\begin{equation*}
\delta _{L}=
\begin{cases}
0 & \text{if} \ q(x+L)=0,  \forall x \in L^{\ast }, \\
1 & \text{otherwise}.
\end{cases}
\end{equation*}
\end{df}

\begin{prop}(\cite[Theorem 3.6.2]{Ni1})\label{cla-2el}
An even indefinite 2-elementary lattice $L$ is determined by 
the invariants $(t_{+},t_{-}, a, \delta _{L}, )$ where the pair $(t_{+},t_{-})$ is the signature of $L$.
\end{prop}

If $L$ is an even hyperbolic 2-elementary lattice, hence its signature is $(1, r-1)$, then 
$L$ is called a \textit{2-elementary lattice of type $(r, a, \delta)$}.

Let $S$ be a $K3$ surface, $\omega _{S}$ a nowhere vanishing holomorphic 2-form on $S$
and $\iota$ a non-symplectic involution on $S$.
Hence $\iota$ is of order 2 and satisfies $\sigma^{\ast} \omega _{S}=-\omega _{S}$.
It is well known that the invariant lattice
\[L:=\{x \in H^{2}(S, \mathbb{Z}) \mid \iota^{\ast}(x)=x \} \]
is an even 2-elementary lattice of signature $(1, r-1)$.
If $L$ is of type $(r, a, \delta)$ then 
a pair $(S, \iota)$ is called a \textit{2-elementary $K3$ surface of type $(r, a, \delta)$}.

\begin{prop}\label{fixedlocus}(\cite[Theorem 4.2.2]{Ni3})
Let $(S, \iota)$ be a 2-elementary $K3$ surface of type $(r, a, \delta)$.
Then the fixed locus $S^{\iota}=\{x \in S \mid \iota(x)=x\}$ of $\iota$ is of the form 
\begin{equation*}
S^{\iota}=
\begin{cases}
\phi  & \text{if $(r, a, \delta)=(10, 10, 0)$}, \\
C^{(1)}\amalg C^{(1)} & \text{if $(r, a, \delta)=(10, 8, 0)$}, \\
C^{(g)} \amalg \mathbb{P}^{1} \amalg \dots \amalg \mathbb{P}^{1} & \text{otherwise},
\end{cases}
\end{equation*}
where $C^{(g)}$ is a genus $g$ curve with $g=(22-r-a)/2$.
Moreover the number of $\mathbb{P}^{1}$ is given by $(r-a)/2$.
\end{prop}

\begin{prop}\label{mt1-1}
Let $S$ be a smooth quartic surface with a quasi-Galois point $P$ 
which is not Galois 
and $\iota$ be a generator of $G[P]$.
Then a pair $(S, \iota)$ is a 2-elementary $K3$ surface of type $(8,8,1)$.
\end{prop}
\begin{proof}
We may assume that $S$ is given by $X^{4}+F_{2}(Y, Z, W)X^{2}+F_{4}(Y, Z, W)=0$
and a generator of $G[P]$ is given by 
$\iota :[X:Y:Z:W]\mapsto [-X:Y:Z:W]$ by Proposition \ref{equation1ko}.
Then we have
\begin{align*}
S^{\iota}&=\{x \in S \mid \iota(x)=x\}\\
&=S \cap (\{X=0\}\amalg \{Y=Z=W=0\})\\
&=\{F_{4}(Y, Z, W)=0\}\\
&=C^{(3)},
\end{align*}
This implies that $r=a=8$ by Proposition \ref{fixedlocus}.
It is known that there do not exist even hyperbolic 2-elementary lattices of type $(8,8,0)$. See also \cite[page 1434]{Ni3}.
Thus $(S, \iota)$ is a 2-elementary $K3$ surface of type $(8,8,1)$.
\end{proof}

\subsection{Quasi-Galois embedding}

\begin{prop}\label{mt1-2}
Let $(S, \iota )$ be a general 2-elementary $K3$ surface of type $(8, 8, 1)$.
Then it gives a smooth quartic surface with a quasi-Galois point.
\end{prop}
\begin{proof}
Note that $S^{\iota}$ consists of $C^{(3)}$ by Proposition \ref{fixedlocus}.
Then there exists the rational map $\phi: S \dashrightarrow \mathbb{P}^{3}$ 
associated to the linear system $|C^{(3)}|$.
For $\phi$ to be an embedding, it is sufficient that
the intersection number of $C^{(3)}$ and any elliptic curve on $S$
is not 2 by \cite[Theorem 5.2 and Theorem 6.1]{SD}.

Let $E$ be an elliptic curve on $S$ and $m$ the intersection number of $C^{(3)}$ and $E$.
We apply the Hurwitz formula for a morphism $E\to E/\iota$ of degree 2.
Since the ramification locus is the intersections of $C^{(3)}$ and $E$, 
we have
\[2\cdot g(E)-2=2(2\cdot g(E/ \iota )-2)+m.\]
This implies that $m$ is 0 or 4.
Thus $\phi$ is an embedding.

Since $C^{(3)}$ is a fixed curve of $\iota$, it preserves $\phi$.
Thus $\iota$ induces a projective transformation $\tilde{\iota}$
which fixes the hyperplane $H$ such that $\phi^{-1}(H)=C^{(3)}$.
By replacing coordinates of $\mathbb{P}^{3}$ such that $H=\{X=0\}$, 
we may assume that 
$\tilde{\iota}$ satisfies $\tilde{\iota}([X:Y:Z:W])=[-X:Y:Z:W]$.
Since $\im \phi$ is invariant for the action of $\tilde{\iota}$,
we find that its equation is of the form in Proposition \ref{equation1ko}.
\end{proof}

The embedding $\phi$ is called a \textit{quasi-Galois embedding},
after the notion of a Galois embedding \cite{Yoshihara-embedding}. 

\begin{rem}
There exists a one-to-one correspondence between
smooth quartic surfaces with an inner (resp. outer) Galois point
and $K3$ surfaces with a certain automorphism of order 3 (resp. 4).
See also \cite{MT1, MT2}.
\end{rem}

\section{$G$-pairs}\label{G-pairs}

In this section, we determine the general equation of a quartic surface with (at least) 2 quasi-Galois points.

\begin{lem}\label{quasi-line}
Let $S$ be a smooth quartic surface, $P$ a quasi-Galois point for $S$.
Then the projective transformation induced by $\tau \in  G[P]$ preserves every line through $P$.
 In particular, if $Q\in\mathbb{P}^{3}$ and $Q\neq P$, then 
 the line $\overline{PQ}$ passing through $P$ and $Q$ satisfies  $\tau(\overline{PQ})=\overline{PQ}$.
\end{lem}

\begin{proof}
We remark that $\tau$ is a projective transformation of $\mathbb{P}^3$ by Lemma \ref{geneprojective}. 
We may choose coordinates such that $P=[1:0:0:0]$ and $\pi_P([X:Y:Z:W])=[Y:Z:W]$.
Write
\[ \tau([X:Y:Z:W])=[L_{0}:L_{1}:L_{2}:L_{3}],\]
where the $L_{i}:=L_{i}(X,Y,Z,W)$ are linear forms. 
Since $\pi_{P}\circ \tau=\pi_{P}$ holds on $S$, we have $[L_{1}:L_{2}:L_{3}]=[Y:Z:W]$ on $S$. 
Note that no nonzero quadratic polynomial vanishes identically on $S$ because it is a hypersurface of degree $4$.
Since the equations $L_{1}Z-L_{2}Y=L_{1}W-L_{3}Y=L_{2}W-L_{3}Z=0$ hold on $S$, 
these quadratic polynomials vanish identically on $\mathbb{P}^{3}$. 
Hence there exists a nonzero constant $c$ such that $(L_{1}, L_{2}, L_{3})=c(Y, Z, W)$.
Thus $\pi_{P}\circ\tau=\pi_{P}$ holds on $\mathbb{P}^{3}$ wherever $\pi_{P}$ is defined. 
Therefore $\tau$ preserves every fiber of $\pi_{P}$, and these fibers are exactly the lines through $P$. 
Hence $\tau(\overline{PQ})=\overline{PQ}$ for every $Q\neq P$.
\end{proof}

The following is defined for curves in \cite{FMT2}.

\begin{df}
Let $S$ be a smooth quartic surface with distinct quasi-Galois points $P_{1}$ and $P_{2}$.
We call the pair $(P_{1}, P_{2})$ a \textit{$G$-pair} if
$\tau_{1}(P_{2})=P_{2}$ and $\tau_{2}(P_{1})=P_{1}$ hold
for generators $\tau_{i} \in G[P_{i}]$.
\end{df}

\begin{lem}\label{qgp-mapsto-qgp}
Let $P$ and $Q$ be quasi-Galois points
for a smooth quartic surface $S$, and $\tau \in G[P]$ be an involution.
The following holds.
\begin{enumerate}
\item The point $\tau (Q)$ is also quasi-Galois.
\item The point $\tau (Q)$ lies on the line passing through $P$ and $Q$.
\end{enumerate}
\end{lem}
\begin{proof}
(1) We may put $P=[1:0:0:0]$ and 
$\tau=\begin{pmatrix} 
-1 & 0 & 0 & 0 \\
0 & 1 & 0 & 0 \\
0 & 0 & 1 & 0 \\
0 & 0 & 0 & 1 
\end{pmatrix}$
by Proposition \ref{equation1ko}.
Since if $Q=[0:a:b:c]$ then $\tau(Q)=Q$ holds, the assertions are immediate.
For $Q=[1:a:b:c]$, we have $\tau (Q)=[1:-a:-b:-c]$,
\[ \pi_{Q}:[X:Y:Z:W] \mapsto [Y-aX : Z-bX : W-cX] \]
and 
\[ \pi_{\tau(Q)}:[X:Y:Z:W] \mapsto [Y+aX : Z+bX : W+cX]. \]
We remark that $\pi_{\tau(Q)}=\pi_{Q}\circ \tau$, and 
there exsits $\tau' \in G[Q]$ such that $\pi_{Q}\circ \tau '=\pi_{Q}$.
If we put $\sigma:=\tau^{-1} \circ \tau ' \circ \tau$ then it is easy to see that
\begin{align*}
\pi_{\tau(Q)}\circ \sigma&=(\pi_{Q}\circ \tau) \circ (\tau^{-1} \circ \tau ' \circ \tau) \\
&=\pi_{Q}\circ  \tau ' \circ \tau\\
&= \pi_{Q}\circ  \tau\\
&=\pi_{\tau(Q)}
\end{align*}
and $\sigma$ is of order 2.
Thus $\tau (Q)$ is a quasi-Galois point.

(2) It follows from Lemma \ref{quasi-line}.
\end{proof}

Let $P, P_{1}$ be distinct quasi-Galois points for $S$ and 
$L:=\overline{PP_{1}}\simeq \mathbb{P}^{1}$ be the line passing through $P$ and $P_{1}$.
Note that involutions $\tau \in G[P]$ and $\tau_{1} \in G[P_{1}]$ preserve $L$.
Hence their restrictions
$\overline{\tau}:=\tau|_{L}$ and $\overline{\tau_{1}}:=\tau_{1}|_{L}$
are automorphisms of $L$. 
We consider their action on the effective divisor 
$D = S \cap L=R_{1}+R_{2}+R_{3}+R_{4}$ of degree 4 on $L$. 
We focus on configurations of $D$:

\begin{lem}\label{4tsuhakasanaru-c1}
If $\operatorname{Supp}D=\{ R_{1}, R_{2}, R_{3}, R_{4} \}$ then
there exists a $G$-pair.
\end{lem}
\begin{proof}
Let $\tau \in G[P]$ be an involution. We set $P_{2}:=\tau (P_{1})$.
We remark that it is a quasi-Galois point for $S$ and 
lies on the line $L$ by Lemma \ref{qgp-mapsto-qgp}.

Assume that $P_{2} \neq P_{1}$. 
Then $\overline{\tau}$ and $\overline{\tau_1}$ are non-trivial. 
Indeed, if $\overline{\tau_{1}}=\id_L$, then $L$ would be pointwise fixed by $\tau_{1}$.
 By Proposition \ref{equation1ko}, after a projective change of coordinates, 
 $\tau_{1}$ is represented by $ [X:Y:Z:W]\mapsto[-X:Y:Z:W]$,
whose fixed locus in $\mathbb{P}^{3}$ is the union of the point $P_{1}$ and 
a plane not containing $P_{1}$.
Since $L$ contains $P_{1}$, it cannot be pointwise fixed by $\tau_{1}$, a contradiction.

We claim that neither $\overline{\tau}$ nor $\overline{\tau_{1}}$ fixes a point of the four-point set 
$\operatorname{Supp}D$.
Suppose, for example, that $\bar{\tau}(R_{i})=R_{i}$ for some $i$. 
Since $\overline{\tau}$ is an involution preserving $\operatorname{Supp}D$, 
its induced permutation on $\operatorname{Supp}D$ has another fixed point $R_{j} \neq R_{i}$. 
On the other hand, $\overline{\tau}$ fixes $P$. 
Thus $\overline{\tau}$ fixes the three distinct points $P, R_{i},R_{j}$. 
Hence $\overline{\tau}=\id_{L}$, which is a contradiction. 
Thus $\overline{\tau}$ has no fixed point on $\operatorname{Supp}D$. 
By the same argument, since $\overline{\tau_{1}}$ fixes $P_{1}$, 
the involution $\overline{\tau_{1}}$ also has no fixed point on $\operatorname{Supp}D$.

The permutations induced by $\overline{\tau}$ and $\overline{\tau_{1}}$ on
$\operatorname{Supp}D$ are both of type $(2,2)$ (i.e., the product of two disjoint transpositions)
 in the symmetric group $\mathfrak{S}_{4}$.
The three permutations of type $(2,2)$, together with the identity, form the Klein four-group $V_{4}$. 
In particular, the induced permutations commute. Hence we have
$\overline{\tau}\circ \overline{\tau_{1}}=\overline{\tau_{1}}\circ \overline{\tau}$ on $\operatorname{Supp}D$.
Since these are projective transformations of $L \simeq \mathbb{P}^{1}$ and
agree on at least three distinct points, they agree on all of $L$.
 Thus $\overline{\tau}\circ \overline{\tau_{1}}=\overline{\tau_{1}}\circ \overline{\tau}$.
Since $P_{2}=\tau(P_{1})$ and $\tau_{1}(P_{1})=P_{1}$, we obtain
\[ \tau_{1}(P_{2})=\tau_{1}(\tau (P_{1}))=\tau(\tau_{1}(P_{1}))=\tau(P_{1})=P_{2}.\]

Let $\overline{P_{2}R}$ be the line passing through $P_{2}$ and a point $R \in \mathbb{P}^{3}$.
Since $\tau(\overline{P_{2}R})=\overline{P_{1}\tau(R)}$ holds 
and $\tau_{1}$ preserve a line passing through $P_{1}$, 
the point $\tau_{1}(\tau(R))$ lies on the same line $\overline{P_{1}\tau(R)}$.
Since $\tau (\overline{P_{1}\tau(R)})=\overline{P_{2}R}$ holds, the point
$\tau(\tau_{1}(\tau(R)))$ lies on the first line $\overline{P_{2}R}$.
This implies that an automorphism $\tau \circ \tau_{1} \circ \tau \in \Aut (S)$ preserve $\pi_{P_{2}}$,
hence we have $\tau \circ \tau_{1} \circ \tau=\tau_{2}\in G[P_{2}]$.
Thus $\tau_{2}(P_{1})=P_{1}$ holds, and  we have a $G$-pair $(P_{1}, P_{2})$.

If $P_{2}=P_{1}$, then it is easy to see the pair $(P, P_{1})$ is a $G$-pair by the same argument as above.
\end{proof}

\begin{lem}\label{4tsuhakasanaru-c2}
The case where $D$ has a unique point of highest multiplicity does not occur.
\end{lem}
\begin{proof}
Put $\overline{\tau}:=\tau |_{L}$ and  $\overline{\tau_{1}}:=\tau_{1} |_{L}$.
These are non-trivial involutions of $L\simeq \mathbb{P}^{1}$. 
Since $\Lin (S)$ is finite by \cite[Theorem 1]{MatsumuraMonsky},  
the subgroup
$\langle \overline{\tau}, \overline{\tau_{1}} \rangle \subset PGL(2,\mathbb{C})$
is also finite.
We first claim that $\overline{\tau}$ and $\overline{\tau_{1}}$ have no common fixed point in $\operatorname{Supp}D$. 
Suppose that a point $R \in\operatorname{Supp}D$ is fixed by both $\overline{\tau}$ and $\overline{\tau_{1}}$.

First, we assume $\overline{\tau} \neq \overline{\tau_{1}}$. 
Taking an affine coordinate $z$ on $L$ such that $R=\infty$, we may write
\[ \overline{\tau}(z)=-z+c,\qquad \overline{\tau_{1}}(z)=-z+d \]
for some $c,d \in \mathbb{C}$. 
Since $\overline{\tau} \neq \overline{\tau_{1}}$, we have $c\neq d$. 
Hence their composition is a non-trivial translation:
$\overline{\tau} \circ \overline{\tau_{1}}(z)=z+(c-d)$,
which has infinite order. This contradicts the finiteness of
$\langle \overline{\tau}, \overline{\tau_{1}} \rangle$.

Next, we assume $\overline{\tau} = \overline{\tau_{1}}$. 
Since $\tau\in G[P]$ and $\tau_{1}\in G[P_{1}]$, these involutions fix $P$ and $P_{1}$, respectively. 
It also fixes $R$ by assumption. Since $P$ and $P_{1}$ are not contained in $S$, the three points
$P$, $P_{1}$, $R$ are distinct. 
Thus $\overline{\tau} = \overline{\tau_{1}}$ fixes three distinct points of $L$, hence it is the identity.
This contradicts the fact that $\overline{\tau}$  and $\overline{\tau_{1}}$ are non-trivial.

Therefore $\overline{\tau}$ and $\overline{\tau_{1}}$ have no common fixed point in $\operatorname{Supp}D$.

Now we suppose that $D$ has a unique point $R_{1}$ of highest multiplicity.
Since both $\overline{\tau}$ and $\overline{\tau_{1}}$ preserve the divisor $D$, they must fix $R_{1}$.
 This contradicts the claim above. 
Hence the configurations
\[ D = 2R_{1} + R_{2} + R_{3}, \quad 3R_{1} + R_{2}, \quad 4R_{1} \] 
do not occur.
\end{proof}

\begin{lem}\label{4tsuhakasanaru-c3}
If $D = 2R_{1} + 2R_{2}$ then one of the following holds:
\begin{enumerate}
\item The pair $(P,P_1)$ is a $G$-pair.
\item The group $\langle \tau,\tau_{1}\rangle$ is isomorphic to the symmetric group $\mathfrak{S}_{3}$. 
In this case, after a change of coordinates, $L=\{Y=Z=0\}$, and $S$ is defined by
\[  (X^{2}-W^{2})^{2}+(X^{2}-W^{2})F_{2}(Y,Z)+W(W^{2}+3X^{2})F_{1}(Y, Z)+F_{4}(Y, Z)=0, \]
where $F_{i}$ is homogeneous of degree $i$ and $F_{1}\neq 0$.
\end{enumerate}
\end{lem} 

\begin{proof}
We remark that $\overline{\tau}:=\tau |_{L}$ and $\overline{\tau_{1}}:=\tau_{1} |_{L}$ 
preserve the set $\{ R_{1}, R_{2} \}$.
We first observe that each of them exchanges $R_{1}$ and $R_{2}$.
Indeed, if $\overline{\tau}$ fixed both $R_{1}$ and $R_{2}$, then
$\overline{\tau}$ would fix the three distinct points $P$, $R_{1}$ and $R_{2}$
of $L\simeq\mathbb P^{1}$, and hence $\overline{\tau}=\id_{L}$, a contradiction.
The same argument applies to $\overline{\tau_{1}}$.

If $\overline{\tau}=\overline{\tau_{1}}$ holds then we have 
$\overline{\tau}(P_{1})=P_{1}$ and $\overline{\tau_{1}}(P)=P$. 
Hence the pair $(P,P_1)$ is a $G$-pair. 
Thus we may assume $\overline{\tau} \neq \overline{\tau_{1}}$.

Let $H$ and $H_{1}$ be the fixed planes of $\tau$ and $\tau_{1}$, respectively.
We claim that $K:=H\cap H_{1}$ is disjoint from $L$. 
Suppose otherwise. 
Then a point of $K\cap L$ is a common fixed point of $\overline{\tau}$ and $\overline{\tau_{1}}$. 
Since two distinct involutions of $\mathbb{P}^{1}$ having a common fixed point 
generate an infinite subgroup of $PGL(2,\mathbb{C})$,
we have $K\cap L=\emptyset$ by the same argument of Lemma \ref{4tsuhakasanaru-c2}.
Consequently, after choosing suitable coordinates, we may write
\[ L=\{Y=Z=0\},\qquad K=\{X=W=0\}, \]
and both $\tau$ and $\tau_{1}$ act trivially on the coordinates $Y, Z$.
Their actions on the two-dimensional vector space with coordinates $X, W$
are linear reflections.

We remark that the group $G:=\langle \tau, \tau_{1} \rangle \subset\Lin(S)$ is finite. 
Put $n:=\ord (\tau \circ \tau_{1})$.
The action of $G$ on the $(X, W)$-space is the dihedral reflection group
\[ I_{2}(n)=\left \langle \tau, \tau_{1} \mid \tau^{2}=\tau_{1}^{2}=(\tau \circ \tau_{1})^{n}=\id  \right \rangle . \]
Choose linear coordinates $u, v$ on this two-dimensional space such that
$\tau \circ \tau_{1}(u,v)=(\zeta_{n} u,\zeta_{n}^{-1}v)$,
where $\zeta_{n}$ is a primitive $n$-th root of unity.
It is easy to see the invariant ring is
$\mathbb{C}[u,v]^{I_{2}(n)}=\mathbb{C}[q_{2}, q_{n}]$, 
where $q_{2}=uv, q_{n}=u^{n}+v^{n}$.

Since $\overline{\tau}$ and $\overline{\tau_{1}}$ exchange $R_{1}$ and $R_{2}$,
their composition $\overline{\tau} \circ \overline{\tau_{1}}$ fixes $R_{1}$ and $R_{2}$. 
Therefore, in the above coordinates, $R_{1}$ and $R_{2}$ are the two eigenpoints
$[1:0]$ and $[0:1]$ of $\overline{\tau} \circ \overline{\tau_{1}}$. 
In particular, $q_{2}=uv$ vanishes simply at $R_{1}$ and $R_{2}$.
We remark that since $Y, Z$ are fixed by $G$, a defining quartic polynomial $F$ of $S$ belongs to
$\mathbb{C}[q_{2},q_{n}, Y, Z]$.
Since $D=S \cap L=2R_{1}+2R_{2}$, the restriction $F|_{L}$ has zeros 
$R_{1}$ and $R_{2}$, each with multiplicity 2.
Hence $F |_{L}$ and $q_{2}^{2}$ have the same zero divisor on $L \simeq \mathbb{P}^{1}$.
Therefore we have $F |_{L}=cq_{2}^{2}$ for some $c \in \mathbb{C}^{\ast}$. 
After rescaling $F$, we may assume $F |_{L}=q_{2}^{2}$.

We determine $n$.
If $n > 4$, no term containing $q_{n}$ can occur in the quartic defining polynomial $F$. 
If $n=4$, the only possible additional term is a constant
multiple of $q_4$, but its coefficient must vanish because
$F |_{L}=q_{2}^{2}$ has exactly the two zeros $R_{1}$ and $R_{2}$, each with multiplicity 2.
Hence in either case the equation of $S$ has the form
\[ q_{2}^{2}+q_{2}A_{2}(Y, Z)+A_{4}(Y, Z)=0.\]
But then it is invariant under the one-dimensional group
\[ (u,v,Y,Z) \mapsto (\lambda u,\lambda^{-1}v,Y,Z), \ (\lambda \in \mathbb{C}^{\ast}) \]
because $q_{2}=uv$ is invariant.
This gives an infinite subgroup of $\Lin(S)$, a contradiction.
If $n=2$, then $\tau \circ \tau_{1}$ acts as $(X, W) \mapsto (-X, -W)$ on the $(X,W)$-space. 
Hence its projective action on $L$ is the identity, so
$\overline{\tau}=\overline{\tau_{1}}$.
This is contrary to our assumption. 
Therefore the only remaining possibility is $n=3$.

We consider the case of $G\simeq I_{2}(3)\simeq \mathfrak{S}_{3}$.
For the dihedral reflection group $I_2(3)$, we may choose coordinates
$u,v$ such that its invariant ring is generated by $uv$ and $u^{3}+v^{3}$.
Setting $u=W+X$ and $v=W-X$, and rescaling the generators if necessary,
$q_{2}=X^{2}-W^{2}$ and $q_{3}=W(W^{2}+3X^{2})$.
Since a $G$-invariant quartic has degree 4, it is necessarily of the form
\[ q_{2}^{2}+q_{2}F_{2}(Y, Z)+q_{3}F_{1}(Y, Z)+F_4(Y, Z)=0.\]
Therefore we have
\[ (X^{2}-W^{2})^{2}+(X^{2}-W^{2})F_{2}(Y, Z)+W(W^{2}+3X^{2})F_{1}(Y, Z)+F_{4}(Y, Z)=0.\]

Finally, we see $F_{1}\neq 0$. 
If $F_{1}=0$ holds then the equation again depends on $X,W$ only through $q_{2}=X^{2}-W^{2}$.
Since 
\[ g_{\lambda} : \left (u, v, Y, Z \right ) \mapsto \left ( \lambda u, \lambda^{-1}v, Y, Z \right ) \]
acts on $S$, $\{g_{\lambda} \}_{\lambda \in \mathbb{C}^{\ast}}$ 
is a one-dimensional subgroup of $\Lin(S)$.
This contradicts the finiteness of $\Lin(S)$.
\end{proof}

\begin{prop}\label{equation2ko}
Let $S$ be a smooth quartic surface with two distinct quasi-Galois points.
Then the homogeneous equation of $S$ is 
\[X^{4}+Y^{4}+
cX^{2}Y^{2}+X^{2}F_{2}(Z, W)+
Y^{2}G_{2}(Z, W)+F_{4}(Z, W)=0, \]
or 
\[  ( X^{2}-W^{2})^{2}+(X^{2}-W^{2})F_{2}(Y,Z)+W(W^{2}+3X^{2})F_{1}(Y, Z)+F_{4}(Y, Z)=0, \]
(up to projective equivalence).
Here $F_{d}$ and $G_{d}$ are homogeneous polynomials of degree $d$, and $c \in \mathbb{C}$.
\end{prop}
\begin{proof}
Assume that $(P_{1}, P_{2})$ is a $G$-pair.
We see that automorphisms which are represented by matrices
\[ \begin{pmatrix} 
-1 & 0 & 0 & 0 \\
0 & 1& 0 & 0 \\
0 & 0 & 1& 0 \\
0 & 0 & 0 & 1
 \end{pmatrix}
 \ \text{and} \ 
\begin{pmatrix} 
1 & 0 & 0 & 0 \\
0 & -1& 0 & 0 \\
0 & 0 & 1& 0 \\
0 & 0 & 0 & 1
 \end{pmatrix}\]
act on $S$ after several coordinate transformations.
We set $F[P]:=\{Q \in \mathbb{P}^{3} \mid \forall \tau\in G[P], \tau (Q)=Q \}$ for a point $P \in \mathbb{P}^{3}$.
Obviously $P_{2}\in F[P_{1}]$ and $P_{1} \in F[P_{2}]$ hold.
Since we may assume $P_{1}=[1:0:0:0]$, we have $F[P_{1}]=\{P_{1}\} \amalg \{X=0\}$.
Thus we put $P_{2}=[0:1:a:b]$.
Moreover, we can assume that $a=b=0$ by taking the coordinate transformation given by 
\[ [X:Y: Z: W] \mapsto [X:Y:Z-aY: W-bY]. \]

Then a generator $G[P_{2}]$ is represented by the matrix
\[A_{2}=\begin{pmatrix} 
1 & 0 & 0 & 0 \\
r_{1} & -1& r_{2} & r_{3} \\
0 & 0 & 1& 0 \\
0 & 0 & 0 & 1
 \end{pmatrix}.\]
 Since the plane $F[P_{2}]\setminus \{P_{2}\}$ is defined by
$r_{1}X-2Y+r_{2}Z+r_{3}W=0$ and $P_{1}\in F[P_{2}]$, we have $r_{1}=0$.
 We recall that a generator $G[P_{1}]$ is represented by the matrix
\[A_{1}=\begin{pmatrix} 
-1 & 0 & 0 & 0 \\
0 & 1& 0 & 0 \\
0 & 0 & 1& 0 \\
0 & 0 & 0 & 1
 \end{pmatrix}.\]
We take the coordinate transformation given by 
\[B=
\begin{pmatrix} 
2 & 0 & 0 & 0 \\
0 & 1& r_{2} & r_{3} \\
0 & 0 & 2 & 0 \\
0 & 0 & 0 & 2
 \end{pmatrix}
.\]
Then we have 
\[
B^{-1}A_{1}B
=
 \begin{pmatrix} 
1/2 & 0 & 0 & 0 \\
0 & 1& -r_{2}/2 & -r_{3}/2 \\
0 & 0 & 1/2 & 0 \\
0 & 0 & 0 & 1/2
 \end{pmatrix}
 \begin{pmatrix} 
-1 & 0 & 0 & 0 \\
0 & 1& 0 & 0 \\
0 & 0 & 1& 0 \\
0 & 0 & 0 & 1
 \end{pmatrix}
 \begin{pmatrix} 
2 & 0 & 0 & 0 \\
0 & 1& r_{2} & r_{3} \\
0 & 0 & 2 & 0 \\
0 & 0 & 0 & 2
 \end{pmatrix}
=
A_{1}
\]
and 
\[B^{-1}A_{2}B=
 \begin{pmatrix} 
1 & 0 & 0 & 0 \\
0 & -1& 0 & 0 \\
0 & 0 & 1& 0 \\
0 & 0 & 0 & 1
 \end{pmatrix}.\]
Hence, if $S$ has a $G$-pair then $S$ is of the form 
\[ X^{4}+Y^{4}+
cX^{2}Y^{2}+X^{2}F_{2}(Z, W)+
Y^{2}G_{2}(Z, W)+F_{4}(Z, W)=0.\]

If $S$ does not have a $G$-pair then $S$ is of the form 
\[  (X^{2}-W^{2})^{2}+(X^{2}-W^{2})F_{2}(Y,Z)+W(W^{2}+3X^{2})F_{1}(Y, Z)+F_{4}(Y, Z)=0, \]
by Lemma \ref{4tsuhakasanaru-c1}--\ref{4tsuhakasanaru-c3}.
\end{proof}

In fact, there exist quartic surfaces that do not possess a $G$-pair.

\begin{claim}\label{G-pair-nashi}
Let $S$ be a smooth quartic surface defined by
\[  (X^{2}-W^{2})^2+W(W^{2}+3X^{2})Y+Y^4+Y^{3}Z+Z^{4}=0.\]
\begin{enumerate}
\item[(I)] $S$ has exactly three quasi-Galois points $[1: 0:0:0], [1:0:0: \pm \eta]$,
where $\eta^{2}=-3$.
\item[(II)] No two of these three points form a $G$-pair.
\end{enumerate}
\end{claim}

\begin{proof}
(I) We set $F=(X^{2}-W^{2})^2+W(W^{2}+3X^{2})Y+Y^4+Y^{3}Z+Z^{4}$ and apply Proposition \ref{polar-hyper}.
We consider the 1-st polar hypersurface and the 3-rd polar hypersurface of $S$ at 
$P:=[a:b:c:d] \in \mathbb{P}^{3} \setminus S$.
A direct computation gives
\begin{align*}
D_{P}F={}&
4aX^{3}+(3b-4d)X^{2}W+3dX^{2}Y-4aXW^{2}+6aXWY\\
&+(b+4d)W^{3}+3dW^{2}Y+(4b+c)Y^{3}+3bY^{2}Z+4cZ^{3}
\end{align*}
and
\[ \frac{1}{6}D_{P}^{3}F=\alpha X+\beta Y+\gamma Z+\delta W, \]
where
\begin{align*}
\alpha  &=2a(2a^{2}+3bd-2d^{2}), \\
\beta  &=3a^{2}d+4b^{3}+3b^{2}c+d^{3}, \\
\gamma  &=b^{3}+4c^{3}, \\
\delta &=3a^{2}b-4a^{2}d+3bd^{2}+4d^{3}. \\
\end{align*}

First we show that $c=0$. Assume $c\neq 0$.
If $\gamma =0$ then $D_P^{3}F$ is independent of $Z$.
However the coefficient of $Z^{3}$ in $D_{P}F$ is $4c\neq 0$.
Hence $D_P^{3}F$ cannot divide $D_{P}F$.
We have $\gamma \neq 0$.
Put $L_{0}:=\alpha X+\beta Y+\delta W$ and $C_{0}:=\left.D_{P}F \right|_{Z=0}$.
Since $D_{P}F=C_{0}+3bY^{2}Z+4cZ^{3}$ and $D_P^{3}F=6(L_{0}+\gamma Z)$,
we have 
\begin{equation}\label{eq1}
\gamma^{3}C_{0}=3b\gamma^{2}Y^{2}L_{0}+4cL_{0}^{3}.
\end{equation}
by Proposition \ref{polar-hyper} and the the factor theorem.
Comparing the coefficients of 
$X^{3}$, $XW^{2}$, $X^{2}Y$, $W^{2}Y$, $XYW$, $W^{3}$ and $Y^{3}$ in (\ref{eq1}),
we obtain
\begin{align}
a\gamma^{3}&=c \alpha^{3}, \label{eq2} \\
-a\gamma^{3}&=3c \alpha \delta^{2}, \label{eq3} \\
d\gamma^{3}&=4c \alpha^{2} \beta, \label{eq4}\\
d\gamma^{3}&=4c \delta^{2} \beta, \label{eq5}\\
a\gamma^{3}&=4c\alpha \beta \delta, \label{eq6}\\
(b+4d)\gamma^{3}&=4c \delta^{3}, \label{eq7}\\
(4b+c)\gamma^{3}&=3b\gamma^{2}\beta+4c \beta^{3}. \label{eq8}
\end{align}
If $a=0$ holds then we have $c \gamma^{3}=0$ by (\ref{eq2}), (\ref{eq4}), (\ref{eq7}), (\ref{eq8}), 
definitions of $\beta$ and $\delta$.
This is a contradiction.
If $a \neq 0$  holds then we have $\alpha^{2}=-3\delta^{2}\neq 0$ by (\ref{eq2}) and (\ref{eq3}).
Note that  (\ref{eq4}) and (\ref{eq5}) imply $\beta(\alpha^{2}-\delta^{2})=0$.
Since $\alpha^{2}=-3\delta^{2}$ and $\delta \neq 0$, we obtain $\beta =0$. 
However, (\ref{eq6}) gives $a \gamma^{3}=0$, a contradiction.
Therefore we have $c=0$.

Next we show $b=0$. 
Assume $b\neq 0$. Then we have $\gamma =b^{3}\neq 0$.
Put again $L_{0}:=\alpha X+\beta Y+\delta W$ and $C_{0}:=\left.D_{P}F \right|_{Z=0}$.
Since $D_{P}F=C_{0}+3bY^{2}Z$ and $D_P^{3}F=6(L_{0}+b^{3}Z)$,
we have 
\begin{equation}\label{eq9}
b^{2}C_{0}=3\gamma^{2}Y^{2}L_{0}.
\end{equation}
by Lemma \ref{polar-hyper} and the the factor theorem.
Comparing the coefficients of $X^{3}$, $X^{2}W$ and $W^{3}$ in (\ref{eq9}),
we obtain
\[ a=0,\qquad 3b-4d=0,\qquad b+4d=0.\]
The last two equalities contradict $b\neq 0$.
Hence we may put $P=[a: 0: 0: d]$.

Since $P$ is not contained in $S$,  $F(P)=(a^{2}-d^{2})^{2} \neq 0$ holds,
especially $A:=a^{2}-d^{2} \neq 0$.
We put $r:=aX-dW$ and $s:=d(X^{2}+W^{2})+2aXW$.
Then we have 
\[ D_{P}F=4(X^{2}-W^{2})r+3Ys \]
 and 
\[ \frac{1}{6}D_{P}^{3}F=4A r+d(d^{2}+3a^{2})Y. \]
If $E:=d(d^{2}+3a^{2})\neq 0$ holds then we have 
$E(X^{2}-W^{2})=3A s$
by Proposition \ref{polar-hyper} and the the factor theorem.
Comparing the coefficients of $X^{2}$ and $W^{2}$ in the equation,
we obtain $E=3A d$ and $-E=3A d$
which is a contradiction for $E\neq 0$. 
Hence we have $E=d(d^{2}+3a^{2})= 0$.

It is easy to see that $r$ divides $s$ in both case $d=0$ and case $d^{2}+3a^{2}=0$.
Thus $D_{P}^{3}F$ divides $D_{P}F$.
This implies that quasi-Galois points of $S$ are precisely 
$[1: 0:0:0], [1:0:0: \pm \eta]$, where $\eta^{2}=-3$.

\vspace{1\baselineskip}

(II) We set $P:=[1: 0:0:0]$, $P_{+}:=[1:0:0: \eta]$ and $P_{-}:=[1:0:0: -\eta]$.
It is easy to see that generators of $G[P]$, $G[P_{+}]$ and $G[P_{-}]$
are given by 
\begin{align*}
\tau &: [X:Y:Z:W] \mapsto [-X:Y:Z:W] \\
\tau_{+} &: [X:Y:Z:W] \mapsto \left[ \frac{X+\eta W}{2}:Y:Z:\frac{-\eta X- W}{2} \right]\\
\tau_{-} &: [X:Y:Z:W] \mapsto \left[ \frac{X-\eta W}{2}:Y:Z:\frac{\eta X- W}{2} \right]
\end{align*}
respectively.
Moreover these satisfy
\[\tau (P_{+})=P_{-}, \ \tau (P_{-})=P_{+},  \quad
\tau_{+} (P)=P_{-}, \ \tau_{+} (P_{-})=P, \quad
\tau_{-} (P)=P_{+}, \ \tau_{-} (P_{+})=P. 
\]
Combined with (I), $S$ does not have a $G$-pair.
\end{proof}

\section{The Fermat quartic surface}\label{Fermat}
Let $S$ be the Fermat quartic surface $X^{4}+Y^{4}+Z^{4}+W^{4}=0$ in $\mathbb{P}^{3}$.

\begin{lem}\label{2ko0}
Let $P:=[1:a:b:c]$ be a quasi-Galois point for $S$.
Then $a=b=0$, $b=c=0$ or $c=a=0$.
\end{lem}
\begin{proof}
We remark that $S$ has exactly four Galois points
$Q_{1}=[1:0:0:0]$, $Q_{2}=[0:1:0:0]$, $Q_{3}=[0:0:1:0]$, $Q_{4}=[0:0:0:1]$
by \cite[Theorem 10, Corollary 6]{Yoshihara--hyper}.
If $P$ is Galois then the assertion holds.

Assume that $P$ is not Galois.
Since a projective linear transformation $\tau=(a_{ij})_{1\leq i,j \leq 4} \in G[P]$ 
preserves the line $\overline{PQ_{k}}$ passing through $P$ and $Q_{k}$ by Lemma \ref{quasi-line},
and $Q_{k}$ is the only Galois point on $\overline{PQ_{k}}$, $\tau$ fixes $Q_{k}$.
Thus if $i\neq j$ then we have $a_{ij}=0$.
Moreover since $\tau$ also fixes $P$ by Corollary \ref{genefixqg}, 
there exists a non-zero constant $k \in \mathbb{C}^{\ast}$ such that
\[\begin{pmatrix} 
a_{11} & 0 & 0 & 0 \\
0 & a_{22}& 0 & 0 \\
0 & 0 & a_{33}& 0 \\
0 & 0 & 0 & a_{44}
 \end{pmatrix}
 \begin{pmatrix} 
1 \\ a \\ b \\ c
 \end{pmatrix}
 =k
  \begin{pmatrix} 
1 \\ a \\ b \\ c
 \end{pmatrix}.
\]
This implies that if $a=b=0, b=c=0$ or $c=a=0$ do not hold then 
$\tau$ is the identity map as a projective transformation.
It contradicts the assumption that $P$ is quasi-Glois.
\end{proof}

\begin{prop}\label{Fermat-basyo}
Points
$[1:\zeta_{4}^{i}:0:0]$, $[1:0:\zeta_{4}^{i}:0]$, $[1:0:0:\zeta_{4}^{i}]$, $[0:1:\zeta_{4}^{i}:0]$, $[0:1:0:\zeta_{4}^{i}]$ 
and
$[0:0:1:\zeta_{4}^{i}]$ ($i=0, 1, 2, 3$)
are quasi-Galois points of $S$.
\end{prop}
\begin{proof}
It follows from Proposition \ref{hobokore}.
\end{proof}

\begin{prop}\label{surface-kara-curve}
If $P:=[1:a:0:0]$ a quasi-Galois point of $S$ then
the Fermat quartic curve $C: X^{4}+Y^{4}+Z^{4}=0$ has also a quasi-Galois point $P'=[1:a:0]$.
\end{prop}
\begin{proof}
We remark that $\pi_{P}$ and $\pi_{P'}$ are given by the followings:
\begin{align*}
\pi_{P}:[X:Y:Z:W]&\mapsto [Y-aX : Z : W]\\
\pi_{P'}:[X:Y:Z]&\mapsto [Y-aX : Z].
\end{align*}
Since there exists a projective automorphism $\tau=(a_{ij})_{1\leq i,j \leq 4}$ of $S$ 
such that $a_{11}=1$ and $\pi_{P}=\pi_{P}\circ \tau$, 
there exists a non-zero constant $k \in \mathbb{C}^{\ast}$ such that
\begin{align*}
k\begin{pmatrix} Y-aX \\ Z \\ W  \end{pmatrix}
&=\begin{pmatrix}
a_{21}X+a_{22}Y+a_{23}Z+a_{24}W-a(X+a_{12}Y+a_{13}Z+a_{14}W) \\ 
a_{31}X+a_{32}Y+a_{33}Z+a_{34}W\\ 
a_{41}X+a_{42}Y+a_{43}Z+a_{44}W
 \end{pmatrix}\\
&=\begin{pmatrix}
(a_{21}-a)X+(a_{22}-aa_{12})Y+(a_{23}-aa_{13})Z+(a_{24}-aa_{14})W\\
a_{31}X+a_{32}Y+a_{33}Z+a_{34}W\\
a_{41}X+a_{42}Y+a_{43}Z+a_{44}W
\end{pmatrix}
\end{align*}

In particular $a_{24}-aa_{14}=a_{34}=a_{41}=a_{42}=a_{43}=0$ holds.
We put $\tau':=\tau_{|C}=(a_{ij})_{1\leq i,j \leq 3}$.
Then it is an automorphism of $C$ and satisfies that $\pi_{P'}\circ \tau '=\pi_{P'}$.
\end{proof}

\begin{thm}
The Fermat quartic surface $S$ has $\delta'[4]=4$ and $\delta'[2]=24$,
hence 28 quasi-Galois points.
\end{thm}
\begin{proof}
We have at least 24 quasi-Galois points by Proposition \ref{Fermat-basyo}
and exactly 4 Galois points by \cite[Theorem 10]{Yoshihara--hyper}.
We see that $S$ has no other quasi-Galois points than these.

If $S$ has a quasi-Galois point $P$ then we may put 
$P=[1:a:0:0]$, $[1:0:b:0]$, $[1:0:0:c]$, $[0:1:b:0]$, $[0:1:0:c]$ or $[0:0:1:c]$ by Lemma \ref{2ko0}.
Let $P=[1:a:0:0]$ be a quasi-Galois point of $S$.
Since $P'=[1:a:0]$ is a quasi-Galois point of the Fermat quartic curve $X^{4}+Y^{4}+Z^{4}=0$
by Proposition \ref{surface-kara-curve}, we have $a=1, \zeta_{4}, -1, \zeta_{4}^{3}$
by \cite[Proposition 3.1 (1), Proposition 3.2]{FMT}.
Thus, quasi-Galois points of $S$  are all 28 points above.
\end{proof}

\end{document}